\documentclass[11pt]{amsart}

\makeatletter
\@namedef{subjclassname@2020}{%
  \textup{2020} Mathematics Subject Classification}
\makeatother

\usepackage{amsmath, amsthm, amsfonts, amssymb,  mathrsfs, graphicx, lscape}
\usepackage{enumitem}
\usepackage[all]{xy}
\usepackage[usenames, dvipsnames]{color}
\usepackage[margin=1in]{geometry} 
\usepackage{tikz}
\usetikzlibrary{matrix,arrows}
\usepackage{tikz-cd}

\numberwithin{equation}{section}
\newtheorem{theorem}[equation]{Theorem}

\newtheorem{proposition}[equation]{Proposition}
\newtheorem{lemma}[equation]{Lemma}
\newtheorem{corollary}[equation]{Corollary}

\theoremstyle{definition}
\newtheorem{rmk}[equation]{Remark}
\newenvironment{remark}[1][]{\begin{rmk}[#1] \pushQED{\qed}}{\popQED \end{rmk}}
\newtheorem{eg}[equation]{Example}

\newtheorem{defn}[equation]{Definition}

\newtheorem{ques}[equation]{Question}

\usepackage[in]{fullpage}
\usepackage[colorlinks=true, pdfstartview=FitV, linkcolor=black, citecolor=black, urlcolor=black]{hyperref}

\usepackage{ stmaryrd }
\usepackage[normalem]{ulem}

\DeclareMathOperator{\codim}{codim}

\DeclareMathOperator{\rank}{rank}

\DeclareMathOperator{\Sym}{Sym}

\DeclareMathOperator{\Gr}{Grass}
\DeclareMathOperator{\GL}{GL}
\DeclareMathOperator{\Sp}{Sp}
\DeclareMathOperator{\SO}{SO}
\DeclareMathOperator{\cha}{char}

\newcommand{\op}{\operatorname}
\newcommand{\oo}{\otimes}
\newcommand{\la}{\lambda}
\newcommand{\al}{\alpha}
\newcommand{\mc}{\mathcal}
\newcommand{\ol}{\overline}

\newcommand{\bb}{\mathbb}

\newcommand{\kk}{\Bbbk}

\begin{document}
\title{Singularities of orthogonal and symplectic determinantal varieties}

\author{Andr\'as Cristian L\H{o}rincz}
\address{University of Oklahoma, Department of Mathematics, Norman, OK 73019}
\email[Andr\'as Cristian L\H{o}rincz]{lorincz@ou.edu}

\begin{abstract}
Let either $\GL(E)\times \op{SO}(F)$ or $\GL(E)\times \op{Sp}(F)$ act naturally on the space of matrices $E\otimes F$. There are only finitely many orbits, and the orbit closures are orthogonal and symplectic generalizations of determinantal varieties, which can be described similarly using rank conditions. In this paper, we study the singularities of these varieties and describe their defining equations. We prove that in the symplectic case, the orbit closures are normal with good filtrations, and  in characteristic $0$ have rational singularities. In the orthogonal case we show that most orbit closures will have similar properties, and determine precisely the exceptions to this.
\end{abstract}


\maketitle

\setcounter{tocdepth}{2}


\section{Introduction}

The irreducible representations of (connected) reductive groups with finitely many orbits have been classified in \cite{kac} (see also \cite{saki}). The study of the singularities and defining equations of orbit closures in such representations generated a lot of interest over the years (see \cite{weymanbook} and references therein). For example, these representations include the spaces of matrices under the action corresponding to row and column operations, yielding the generic, symmetric, and skew-symmetric determinantal varieties.

Among Kac's list of irreducible representations with finitely many orbits, there are only $3$ infinite series involving more than one parameter. One of them is the generic determinantal case, and the other two are the main subject of this paper  -- in fact, all the generic determinantal varieties appear also within these two cases as orbit closures. Therefore, the two cases under consideration in this paper are the largest series of irreducible representations with finitely many orbits.

Most representations from Kac's list can be constructed by the Vinberg method, using Dynkin diagrams and distinguished nodes. In a program lead by Kra\'skiewicz and Weyman \cite{kraswey1, kraswey2}, the systematic study of singularities and free resolutions of orbit closures of Vinberg representations has been undertaken. Our two representations can be described by the Dynkin diagrams of type $B, C, D$ together with a choice of a node.

\bigskip

We now describe our main setup. Let $E, F$ be vector spaces over an algebraically closed field $\kk$, of respective dimensions $e\geq 1, f \geq 3$. 

Suppose that $F$ is equipped with a non-degenerate bilinear form $\beta$ that is either symmetric (when we assume $\cha \kk \neq 2$) or alternating (with $f$ even). Such a form yields an isomorphism $j: F \cong F^*$. For any map $\phi: E^* \to F$, we consider the composition
\[
 E^* \xrightarrow{\;\; \phi \;\;} F \xrightarrow{\;\; j \;\;} F^* \xrightarrow{\;\; \phi^* \;} E,
 \]
 which yields a $\GL(E)$-equivariant morphism 
 \begin{equation}\label{eq:psi}
 \psi: E\oo F \to E\oo E.
\end{equation}
In fact, when the form $\beta$ is symmetric, $\op{im} \psi \subset \op{S}_2(E)$ (the space of symmetric tensors or matrices), and when $\beta$ is alternating $\op{im} \psi \subset \bigwedge^2 E$ (the space of skew-symmetric tensors or matrices).
 
For a pair of integers $(r_1,r_2)$ satisfying 
\begin{equation}\label{eq:condi}
0\leq r_2 \leq r_1 \leq e,  \quad 2 r_1 - r_2 \leq f, \quad \mbox{if } \beta \mbox{ alternating then } r_2 \mbox{ even},
\end{equation}
and consider the closed subvarieties in $E\oo F$ defined by conditions on the rank and isotropic rank of an element $\phi$:
\begin{equation}\label{eq:rank}
\rank \phi \leq r_1, \quad \,\rank \psi(\phi) \leq r_2.
\end{equation}
For almost all pairs $(r_1,r_2)$, these will define irreducible subvarieties, denoted by $\ol{O}_{r_1,r_2}$; however, when $\beta$ is symmetric and $(r_1, r_2) = (f/2, 0)$ (with $f$ even), they will have $2$ irreducible components (that are isomorphic to each other), in which case we denote by $\ol{O}_{r_1,r_2}$ either one of them. 
 
As the notation suggests, the varieties $\ol{O}_{r_1,r_2}$ are precisely the closures of the orbits $O_{r_1,r_2}$ under action of either the group $\GL(E)\times \SO(F)$ (when $\beta$ is symmetric) or $\GL(E)\times \Sp(F)$ (when $\beta$ is alternating). In fact, these are all the orbits under these actions (with the caveat mentioned above in the orthogonal case), and they can be described as those matrices for which we have equalities in (\ref{eq:rank}).

\medskip

The geometry of these orbit closures has been studied in \cite{lovett05, lovett07}, especially from the point of view of free resolutions. In this paper, we strengthen the results found there regarding the singularities of orbit closures, and also describe generators of their defining ideals. Furthermore, our methods work in positive characteristic as well.

There are many interesting special cases appearing between these varieties. When $r_1=r_2$, the isotropic condition in (\ref{eq:rank}) is superfluous and we obtain all of the generic determinantal varieties, as we mentioned above. When $r_2=0$ and $r_1$ maximal satisfying the requirements (\ref{eq:condi}), then the isotropy conditions in (\ref{eq:condi}) yield precisely the (reduced) nullcones of the action of $\SO(F)$ and $\Sp(F)$, respectively, which were studied more recently in \cite{kraft,vaib1, vaib2}.

The following are the main results of the paper. We begin with the case when $\beta$ is alternating, since then the results are uniform.

\begin{theorem} \label{mainsymp}
All $\GL(E)\times \Sp(F)$-orbit closures $\ol{O}_{r_1, r_2}$ are normal, and when $\cha \kk = 0$ have rational singularities (thus are Cohen--Macaulay). Furthermore, the minors and Pfaffians corresponding to the conditions (\ref{eq:rank}) give (good) generators of their defining ideals.
\end{theorem}

When $\beta$ is symmetric, there are some exceptional cases.

\begin{theorem} \label{mainorth}
The $\GL(E)\times \SO(F)$-orbit closures in $\ol{O}_{r_1, r_2}$ are normal if and only if either $r_2\neq 2 r_1 - f$ or $r_2=0$ or $r_2=e$. When  $\ol{O}_{r_1, r_2}$ is not normal, it is Cohen--Macaulay if and only if $r_1=e$. The minors corresponding to the conditions (\ref{eq:rank}) give (good) generators of the defining ideals of $\ol{O}_{r_1, r_2}$, unless $(r_1,r_2) = (f/2, 0)$.
\end{theorem}

We give also (good) defining generators in the latter case $(r_1,r_2) = (f/2, 0)$ (with $f$ is even) in Proposition \ref{prop:eqred}. Here, some additional equations are necessary since the rank conditions yield two irreducible components, as discussed above.

\medskip

Thus, in both the symplectic and orthogonal cases we determine (good) defining equations for all orbit closures and classify the normal orbit closures in arbitrary characteristic (not two for the orthogonal case). We further prove that some special orbit closures are strongly $F$-regular (hence, Cohen--Macaulay) in positive characteristic.

\smallskip

One of our motivations for this work comes from the point of view of $\mathcal{D}$-modules. The paper \cite{lorper} studies the equivariant $\mathcal{D}$-modules on $E\oo F$, together with related invariants (such as local cohomology), in the orthogonal case when $f=4$. This is part of a larger program of classification of equivariant $\mathcal{D}$-modules (and related invariants) for all irreducible representations with finitely many orbits. To proceed with this program in the larger cases, understanding the singularities of orbit closures becomes necessary (cf. \cite[Section 6]{lorper}).

\medskip

The two main methods that we employ are both relative, i.e. they allow carrying over singularities from smaller varieties to bigger ones. They can be utilized in various other settings beyond the varieties $\ol{O}_{r_1,r_2}$ considered in this paper.

\smallskip

\textbf{Proof outline:} 
We can interpret the map (\ref{eq:psi}) above as an affine quotient map, via the fundamental theorems of invariant theory for the orthogonal or symplectic groups. We use this map to lift various properties from the $\GL(E)$-orbit closures in the target spaces $\op{S}_2(E)$ or $\bigwedge^2 E$ to their preimages, and thus obtain the results for orbit closures of the form $\ol{O}_{e,r_2}$, (i.e. with full rank $r_1=e$). 

Once we have the results for the full rank case, we transfer its properties using the collapsing method \cite{kempf76, me} to general orbit closures $\ol{O}_{r_1, r_2}$.
\section{Preliminaries}

Throughout, we work over an algebraically closed field $\kk$.  For an algebraic group $G$, all $G$-modules considered are assumed to be rational of countable dimension. Unless otherwise stated, throughout schemes are Noetherian, $\kk$-algebras finitely generated and commutative.


\medskip

Let $G$ be a connected reductive group over $\kk$, $B$ a Borel subgroup and $U$ its unipotent radical. We fix a maximal torus $T \subset B$, and denote by $X(T)$ its group of characters.

For $\la \in X(T)$, we put $\mc{L}(\la)$ to be the sheaf of sections of the line bundle $G\times_B \kk_{-\la}$, where $\kk_{-\la}$ is associated the $1$-dimensional representation of $B$.

A weight $\la \in X(T)$ is dominant if $\langle \la , \al^\vee \rangle \geq 0$, for all simple roots $\al \in S$. The set of dominant weights is denoted by $X(T)_+$. For $\la \in X(T)_+$, we call the space of sections
\[\nabla_G(\la):=H^0(G/B, \mc{L}(\la)),\]
a \emph{dual Weyl module}. It has lowest weight $-\la$ and highest weight $-w_0 \cdot \la$. The module $\Delta_G(\la) = \nabla_G(\la)^*$ is called a \emph{Weyl module}, that has a non-zero highest weight vector of weight $\la$, and this generates $\Delta_G(\la)$ as a $G$-module (cf. \cite[Section 2.1, page 64]{brionkumar}).

Take a (possibly infinite-dimensional) $G$-module $V$. Following Donkin \cite{donkin}, an ascending exhaustive filtration 
\[0 = V_0 \subset V_1 \subset V_2 \subset \dots\]
of $G$-submodules of $V$ is a \textit{good filtration} of $V$, if each $V_i/V_{i-1}$ is isomorphic to a dual Weyl module. 

A finite-dimensional $G$-module $W$ is \textit{good} if $\Sym_d W^*$ has a good filtration for all $d\geq 0$.  Similarly, we call an affine $G$-scheme $X$ good (or $G$-good) if $\kk[X]$ has a good filtration. 

\smallskip

If $\cha \kk = 0$, then all $G$-modules have good filtrations. We start with the following basic result \cite[Propositions 3.1.1 and 3.2.4, Corollary 3.2.5]{donkin}.

\begin{lemma}\label{lem:seq}
Consider an exact sequence of $G$-modules
\[ 0 \to A \to B \to C \to 0.\]
\begin{enumerate}
\item Assume that $A$ has a good filtration. Then if $B$ (resp. $C$) has a good filtration, then so does $C$ (resp. $B$).
\item Assume the sequence splits. Then $B$ has a good filtration if and only if both $A$ and $C$ have good filtrations.
\end{enumerate}
\end{lemma}

Next, we recall an important result of Donkin \cite{donkin} and Mathieu \cite[Theorem 1]{mathieu}.

\begin{theorem}\label{thm:mathieu}
If $M$ and $N$ are $G$-modules with good filtrations, then $M\otimes_\kk N$ has a good filtration.
\end{theorem}

For the following, see \cite[Section 4]{andjan}.

\begin{lemma}\label{lem:good}
Let $V,W$ be finite-dimensional $G$-modules. If $\bigwedge V$ and $ \bigwedge W$ have good filtrations, then $V\otimes W$ is good.
\end{lemma}

\smallskip

By a $G$-algebra $R$ we mean a $\kk$-algebra that is rational as a $G$-module such that the maps $\kk \to R$ (viewing $\kk$ as a trivial $G$-module) and $R \oo R \to R$ (given by multiplication) are $G$-equivariant. 

\smallskip

Following \cite[Definition 2.13]{me}, we have the following notion on generators of modules.

\begin{defn}\label{def:goodgen}
Let $R$ be a $G$-algebra with a good filtration and $M$ a finitely generated $(G,R)$-module. We say that a finite set of equations $\mc{P} \subset M$ are \textit{good generators} of $M$ if the following hold for $V_{\mc{P}} := \op{span}_{\kk} \mc{P}  \, \subset M$:
\begin{enumerate}
\item $V_{\mc{P}}$ is a $G$-module with a good filtration;
\item The kernel of the onto multiplication map $m_{\mc{P}}: R \oo V_{\mc{P}} \to M$ has a good filtration.
\end{enumerate}
By abuse of terminology, in such case we will also say that the $G$-submodule  $V_{\mc{P}} \subset M$ gives good generators for $M$.

When $Y$ is a good affine $G$-scheme and $X\subset Y$ a closed $G$-stable subscheme with ideal $I_X \subset \kk[Y]$, we also call a set of good generators $\mc{P} \subset I_X$ good defining equations of $X$ in $Y$.
\end{defn}

\medskip

Note that (2) can be replaced by the equivalent condition of surjectivity of the multiplication map on the level of $U$-invariants (see \cite[Section 2.4]{me}). Further, there exist good defining equations for $X\subset Y$ if and only if $(Y,X)$ is a good pair \cite[Lemma 2.14]{me}. We will see that any module with a good filtration has good generators (see Proposition \ref{prop:goodgen} below).

We mention a basic result that was implicitly used in the example from \cite[Section 4.1]{me}. We use the same notation as in the definition above.

\begin{lemma}\label{lem:throwgens}
Let $V\subset M$ give good generators with $V\cong V_1 \oplus V_2$ as $G$-modules. If $V_1$ generates $M$, then it also gives good generators for $M$.
\end{lemma}

\begin{proof}
By Lemma \ref{lem:seq} (2), both $V_1$ and $V_2$ have good filtrations. Let $K$ (resp. $K_1$) be the kernel of the multiplication map $m$ (resp. $m_1$). Then we have the following commutative diagram with exact rows and columns:
\[\xymatrix@R1.3pc{
& 0 \ar[d] &  0 \ar[d] & & \\
0 \ar[r] & R\oo V_2 \ar[r] \ar[d] & R\oo V_2 \ar[r] \ar[d] & 0  \ar[d] &\\
0 \ar[r] & K \ar[r] \ar[d] & R\oo V \ar[d] \ar[r]^{\quad m} & M \ar[r] \ar[d]& 0\\
0 \ar[r] & K_1 \ar[r] \ar[d] & R\oo V_1 \ar[d] \ar[r]^{\quad m_1} & M \ar[r] \ar[d] & 0 \\
 & 0 & 0 & 0 &
}\]
By definition, $K$ has a good filtration, and so does $R\oo V_2$ by Theorem \ref{thm:mathieu}. Hence, by Lemma \ref{lem:seq} (1) we obtain that $K_1$ has a good filtration.
\end{proof}

Lastly, we mention the following result that we use to find good defining equations in the exceptional cases.

\begin{lemma}\label{lem:reducible}
Let $I_1, I_2 \subset R$ be two $(G,R)$-ideals. Assume that there exists $G$-submodules $V_{i}\subset I_i$ such that $V_{1} \cap V_{2}$ gives good generators for $I_1 \cap I_2$, and of $V_{1} + V_{2}$ good generators for $I_1+I_2$. Then $V_{i}$ gives good generators for $I_i$.
\end{lemma}

\begin{proof}
From the short exact sequence 
\[0 \to V_{1} \cap V_{2} \to V_{1} \oplus V_{2} \to V_{1} + V_{2} \to 0\]
we see that $V_{i}$ has a good filtration for $i=1,2$, by Lemma \ref{lem:seq}. Based on this sequence, we have the following commutative diagram with exact rows and columns:

\[
\xymatrix@R1.4pc{
& 0 \ar[d] &  0 \ar[d] & 0 \ar[d] & \\
0 \ar[r] & K \ar[r] \ar[d] & K_1 \oplus K_2 \ar[r] \ar[d] & K'  \ar[d] \ar[r] & 0\\
0 \ar[r] &  R\oo (V_1 \cap V_{2}) \ar[r] \ar[d] & R\oo V_{1} \bigoplus R\oo V_{2} \ar[d]^{m_{1}\oplus m_{2}} \ar[r] & R\oo (V_{1} + V_{2}) \ar[r] \ar[d]& 0\\
0 \ar[r] & I_1 \cap I_2 \ar[r] \ar[d] & I_1 \oplus I_2 \ar[d] \ar[r] & I_1 + I_2 \ar[r] \ar[d] & 0 \\
 & 0 & 0 & 0 &
}\]
Since by definition the kernels $K$ and $K'$ have good filtrations, so do the kernels $K_i$ by Lemma \ref{lem:seq}.
\end{proof}



\section{Generic perfection and good filtrations}

Throughout, a graded algebra $R$ is a graded $\kk$-algebra of the form $R=\bigoplus_{i\geq 0}  R_i$ with $R_0 = \kk$ and each $R_i$ finite dimensional over $\kk$. By a graded $G$-algebra we mean a graded algebra $R$ that is a $G$-algebra such that each graded piece $R_i$ is a $G$-submodule. We similarly define a $(G,R)$-module to be an $R$-module and a $G$-module such that these structures are compatible in the obvious way. Again, for a graded $(G,R)$-module we require the graded pieces to be $G$-modules.

\smallskip

As before, let $G$ is a connected reductive group. 
We recall that for a $G$-algebra $R$ and a finitely generated $(G,R)$-module $M$,  $R^U$ is a finitely generated $\kk$-algebra and $M^U$ is a finitely generated $R^U$-module (see  \cite[Theorems 9.4 and 16.8]{grossbook}).

\begin{proposition}\label{prop:goodgen}
Let $R$ be $G$-algebra with a good filtration, $M$ a $(G,R)$-module that is finitely generated over $R$. Let $\mc{S}$ be a finite set of generators of the $R^U$-module $M^U$. Assume that there exists a $G$-module $V \subset M$ with a good filtration such that $\mc{S} \subset V$. Then $M$ has a good filtration, and $V$ yields good generators for $M$.
\end{proposition}

\begin{proof}
We can assume $V$ to be finite dimensional. We consider the multiplication map
\[ \pi: \, R \oo V \, \rightarrow \, M, \]
which is a map of finitely generated $(G,R)$-modules. $R \oo V$ has a good filtration, by Theorem \ref{thm:mathieu}. We have an induced map $\pi^U:  (R \oo V)^U \rightarrow M^U$. Then $\pi^U$ is surjective by construction, since its restriction to the subspace $R^U \oo V^U$ is surjective as $\mc{S} \subset V^U$. By \cite[Lemma 2.11]{me} this implies that $\pi$ is surjective, and both $M$ and $\ker \pi$ have a good filtration.
\end{proof}

In particular, in the graded setting the above shows that one needs to check the property of good filtrations only at the degrees of the minimal generators of the $R^U$-module $M^U$.

\begin{corollary}\label{cor:res}
Let $R$ be $G$-algebra with a good filtration, $M$ a $(G,R)$-module with a good filtration, finitely generated over $R$ and of finite projective dimension $c$. Then $M$ has a $G$-equivariant projective resolution $P_\bullet$ of length $c$ such that all terms and syzygies have good filtrations.
\end{corollary}

\begin{proof}
We will construct $P_\bullet$ inductively. Suppose that for some $0\leq i \leq c-1$, we have a finitely generated $(G,R)$-module $M_i$ that is the $i$th syzygy, having a good filtration (with $M_0=M$). Pick a $G$-module $V_i \subset M_i$ as in Proposition \ref{prop:goodgen} giving good generators of $M_i$. Then the multiplication map
\[ \pi_i: \, R \oo V_i \, \rightarrow \, M_i\]
is surjective and $\ker \pi_i$ has a good filtration, so we put $P_i:=R \oo V_i$ with the next syzygy $M_{i+1} = \ker \pi_i$. For the last term of the resolution we set $P_c:=M_c$, which must be a projective $R$-module by simple homological considerations.
\end{proof}

\begin{remark}\label{rem:res}
\begin{itemize}
\item[(a)] More precisely, the above shows that $P_\bullet$ can be chosen with the property that for each $0\leq i \leq c-1$, there exists a finite dimensional $G$-module $V_i$ with a good filtration such that $P_i \cong R \oo_{\kk} V_i$ as $(G,R)$-modules.
\item[(b)] It is easy to see that if $R$ and $M$ are additionally graded then the resolution $P_\bullet$ as well as the isomorphisms $P_i \cong R \oo_{\kk} V_i$ above can be chosen to be also graded.
\end{itemize}
\end{remark}

By Remark \ref{rem:res} (a), all but the last term of the complex can be chosen to have particularly simple $(G,R)$-module structures. The next result shows, in particular, that the last term can be realized as extensions of such modules (cf. \cite[Chapter IV, Lemma 1.1.4]{hashibook}).

\begin{lemma}\label{lem:last}
Let $R$ be a graded polynomial $G$-algebra, and $F$ a finitely generated graded free $(G,R)$-module (i.e. free as an $R$-module). Then $F$ a has finite exhaustive filtration $F_\bullet$ of graded free $(G,R)$-modules with short exact sequences
\[0\to F_{i-1} \to F_i \to R\oo V_i \to 0 \quad \mbox{for } i\geq 0, \, \mbox{with } F_{-1}=0,\]
for some finite dimensional $G$-modules $V_i$. Furthermore, if $F$ has a good filtration then one can choose all $V_i$ to have good filtrations.
\end{lemma}

\begin{proof}
Since $R$ is a graded polynomial ring, any graded free module is a finite direct sum of the shifted rings $R(d)$, with $d\in \mathbb{Z}$.

We will additionally see that each $F/F_i$ is a graded free $(G,R)$-module. We proceed by induction on $i$, the case $i=-1$ being vacuous. Let $i\geq 0$, so that $F/F_{i-1}$ is a free graded free $(G,R)$-module. Let $V_i$ be the non-zero graded piece of $F/F_{i-1}$ of lowest possible degree. Being a graded piece, $V_i$ must be a $G$-module. By our initial remark, we have an inclusion map $R\oo V_i \to F/F_{i-1}$ of free $(G,R)$-modules, which splits over $R$ by construction. We set $F_i$ to be $R\oo V_i = F_{i}/F_{i-1} \subset F/F_{i-1}$, which is free as an $R$-module, and so is $F/F_i$ due to the splitting.

Finally, if $F$ has a good filtration then by induction we assume $F_{i-1}$ has one, therefore $F/F_{i-1}$ as well by Lemma \ref{lem:seq} (1). Since by construction $V_i$ is a direct summand of $F/F_{i-1}$ as a $G$-module, $V_i$ has a good filtration by Lemma \ref{lem:seq} (2). By Theorem \ref{thm:mathieu} $R\oo V_i$ has a good filtration and so by Lemma \ref{lem:seq} (1), $F_i$ also has a good filtration, thus finishing the proof by induction.
\end{proof}

Below a finite dimensional $G$-module $V$ will be viewed as $G\times \kk^*$-module via some power $t\in \bb{Z}_{>0}$ of the scalar multiplication, i.e. $c \cdot v = c^t v$, for $c \in \kk^*$ and $v\in V$. The following is one of the main tools we use to lift singularities and good filtrations (we note that parts (1) and (2) are well-known).

\begin{proposition}\label{prop:main}
Let $X$ be a Cohen--Macaulay affine noetherian $G\times \kk^*$-scheme, $V$ a finite dimensional $G$-module, and $f: X \to V$ a $G \times \kk^*$-equivariant morphism. Assume that both $V$ and $X$ are $G$-good. Let $Z\subset V$ be a closed subscheme that is $G \times \kk^*$-stable and Cohen--Macaulay, and consider $Y:=f^{-1}(Z)$ the (scheme-theoretic) preimage. If $\codim_Y X = \codim_Z V$, then:
\begin{enumerate}
\item $Y$ is Cohen--Macaulay.
\item Assume that $X$ is Gorenstein. Then $Y$ is Gorenstein if and only if $Z$ is so.
\item If $Z$ is good, then $Y$ is good.
\item If $\mc{P}$ is a set of good (resp. minimal) defining equations of $Z$, then $f^{*}(\mc{P})$ is a set of good (resp. minimal) defining equations of $Y$.
\end{enumerate}
\end{proposition}

\begin{proof}
Let $R=\kk[V]$, $X=\op{Spec} S$, and $I\subset R$ the ideal of $Z$. Then $f^* : R \to S$ is an injective $G$-equivariant morphism of graded $G$-algebras. Put $c=\codim_Y X = \codim_Z V$. We will repeatedly use the Generic Perfection Theorem \cite{EN} (see also \cite[Theorem 1.2.14]{weymanbook}) to prove these statements.

Let $F_\bullet$ be the minimal free resolution of $R/I$, which has length $c$. By the Generic Perfection Theorem, $S\oo_R F_\bullet$ is a minimal free resolution of $S\oo_R (R/I)$ of length $c$. This proves (1), and part (2) follows also since the rank of the last term is preserved (see \cite[Proposition 1.2.10]{weymanbook}). The claim in part (4) regarding the minimality of the defining equations also follows for analogous reasons.

To prove (3), we consider now a graded free resolution of the form $P_\bullet$ as in Corollary \ref{cor:res} -- see Remark \ref{rem:res} (note that the last term $P_c$ is also (graded) free). By the Generic Perfection Theorem, $S\oo_R F_\bullet$ is a free resolution $(G,R)$-modules of $S\oo_R (R/I)$. 

Take any term $A=P_i$. By Lemma \ref{lem:last}, $A$ a has finite filtration $A_\bullet$ of free $(G,R)$-modules with factors of the form $R\oo V_j$, for $G$-modules $V_j$ that have good filtrations. Then $S \oo_R A_\bullet$ is a filtration $S\oo_R A$ by $(G,S)$-modules, with factors isomorpic to $S\oo V_j$. By Theorem \ref{thm:mathieu}, these factors have good filtrations. Then by induction, we see by Lemma \ref{lem:seq} (1) that $S \oo_R A$ has a good filtration as well.

Thus,  $S \oo_R P_\bullet$ is an equivariant resolution of $S\oo_R (R/I)$ such that each term has a good filtration. Using Lemma \ref{lem:seq} (1) repeatedly, we see that all the syzygies have good filtrations. In particular, so does $S\oo_R (R/I)$.

For part (4) regarding good defining equations, set $V_{\mc{P}} = \op{span}_{\kk} \mc{P} \subset I$. Since $\mc{P}$ form good generators, we note that, by construction, the resolution $P_\bullet$ above can be chosen such that $P_0 =R$ and $P_1 = R\oo V_{\mc{P}}$ (see Proposition \ref{prop:goodgen} and Corollary \ref{cor:res}). By the above, since both the terms and the syzygies in  $S \oo_R P_\bullet$ will have good filtrations, the kernel of the onto map $S \oo V_{\mc{P}} \to S$ will have a good filtration. Thus, the set $f^{*}(\mc{P})$ forms good defining equations for $Y$.
\end{proof}


%

\section{Symplectic case}

Throughout in this section we let $G=\GL(E)\times \Sp(F)$ act naturally on $X=E\oo F$. We start by noting that $X$ is good (i.e. $\kk[X]$ has a good filtration), which follows from Lemma \ref{lem:good} and \cite[Section 4.9]{andjan}. We recall (see \cite[Corollary 2.4]{lovett05})
\begin{equation}\label{eq:codim1}
\codim_X \ol{O}_{r_1,r_2} =  (e-r_1)(f-r_1)+ \binom{r_1-r_2}{2}.
\end{equation}

We begin with an easy result.

\begin{lemma}\label{lem:smooth1}
Let $J$ be a $b\times b$ invertible skew-symmetric matrix, and $A$ an $a\times b$ matrix with $\rank A=a$. Then any $a\times a$ skew-symmetric matrix $S$ can be written in the form $S=A J B^t + B J A^t$, for some $a \times b$ matrix $B$.
\end{lemma}

\begin{proof}
We can write $S=X-X^t$, for some $a \times a$ matrix $X$. Let $Y$ be a left inverse of the matrix $J A^t$, i.e. $Y J A^t= I_a$. We can then put $B=XY$.
\end{proof}

We start with the special orbit closures of full rank $r_1=e$.

\begin{proposition}\label{prop:sympinv}
Let $e < f$. Then for any even $r$ with $2e-f\leq r\leq e$, the variety $\ol{O}_{e, r}$ is normal, Gorenstein, and if $\cha \kk =0$ has rational singularities. Furthermore, the isotropic rank conditions in (\ref{eq:rank}) given by the Pfaffians yield good defining equations of $\ol{O}_{e, r}$ in $X$.
\end{proposition}

\begin{proof}
By the fundamental theorem of invariant theory for the symplectic group \cite{deconproc}, the map (\ref{eq:psi}) is in fact the onto map
\[\psi: X \longrightarrow  X/\!\!/\Sp(F) \cong \bigwedge^2 E.\]
We view $\bigwedge^2 E$ also as a $\kk^*$-module via the induced action from $E$ by scalar multiplication. We further view it as a trivial $\Sp(F)$-module. Then $\psi$ is a $G\times \kk^*$-equivariant map. 

Let $O_r \subset \bigwedge^2 E$ be the locally closed subvariety of matrices of rank $r$. Then the closure $\ol{O}_r$ is Gorenstein (see \cite[Theorem 17]{kleplaks}) and the set $\mc{P}$ of $(r+2) \times (r+2)$ Pfaffians give good defining equations as discussed in \cite[Section 4.1]{me} (see also Lemma \ref{lem:throwgens}).


Let $J$ be the $f\times f$ skew-symmetric matrix of the alternating form $\beta$ in some basis. We realize $\psi$ as a map taking $e\times f$-matrices to $e\times e$ skew-symmetric matrices via
\[\psi: \, B \mapsto B \cdot J \cdot B^t.\]
Then the differential $d\psi_A: X \to \bigwedge^2 E$ at a matrix $A \in X$ is given by
\[d\psi_A: \, B \mapsto AJB^t + BJA^t.\]
Consider the open $U \subset X$ of elements of full rank $e$. By Lemma \ref{lem:smooth1}, the differential of $\psi_{| U} : U \to \bigwedge^2 E$ is onto at each point. Hence, $\psi_{| U}$ is a smooth morphism (\cite[Proposition III.10.4]{hartshorne}).

Set $Y_r := \psi^{-1}(\ol{O}_r)$ as a scheme. Clearly, the underlying variety of $Y_r$ is $\ol{O}_{e,r}$. Since  $\psi_{| U}$ is a smooth morphism, we have $\codim_X {\ol{O}_{e,r}} = \codim_{ \bigwedge^2 E} {\ol{O}_r}$ (this also follows by an easy dimension comparison). By Proposition \ref{prop:main}, $Y_r$ is Gorenstein with good defining equations $\psi^*(\mc{P})$. To show that  $Y_r=\ol{O}_{e,r}$ we are left to show that $Y_r$ is generically reduced, for which it is enough to see that the scheme $\psi^{-1}_{|U}(O_r)$ is smooth. This follows by \cite[Proposition III.10.1]{hartshorne} since $\psi_{| U}$ and $O_r$ are smooth.

To obtain normality, by (\ref{eq:codim1}) we see that $\ol{O}_{e,r}$ is regular in codimension $1$, since we have $\codim_{\ol{O}_{e,r}} \ol{O}_{e-1,r} \geq 2$ and $\codim_{\ol{O}_{e,r}} \ol{O}_{e,r-2} \geq 2$. Thus the claim follows by Serre's criterion \cite[II.Theorem 8.22A]{hartshorne}.

We are left to prove the claim on rational singularities in $\cha \kk=0$. By \cite[Proposition 2.9]{lovett05}, we can write the $\ol{O}_{e,r}$ as the image of a (not necessarily birational) collapsing map $Z \to X$, where $Z$ is the total space of a vector bundle over an isotropic Grassmannian. This yields a complex $F_\bullet$ as in \cite[Theorem 5.1.2]{weymanbook}. It follows from \cite[Theorem 4.2 and Corollary 4.5]{lovett05} (see also \cite[Theorem 4.6]{lovett05}) that $F_0 = \kk[X]$, and $F_i = 0$, for $i<0$.  We conclude that $\ol{O}_{e,r}$ has rational singularities as in \cite[Theorem 2.1 (b)]{lorwey}.
\end{proof}

\begin{remark}\label{rem:errors}
We note some similarities of the above with the proof of \cite[Theorem 4.6]{lovett05} (which is in characteristic zero). Nevertheless, in the latter there are some gaps in the proof of showing that the preimage is reduced and normal, and there are similar gaps in the orthogonal case \cite[Theorems 4.9 and 4.10]{lovett05}. Moreover, the claims in \cite[Proposition 4.3]{lovett07} regarding the Gorenstein property are incorrect, e.g. by the above result.
\end{remark}

\begin{theorem}\label{thm:mainsymp}
The varieties $\ol{O}_{r_1, r_2}$ are normal, and when $\cha \kk =0$ have rational singularities. Moreover, the minors and Pfaffians corresponding to the rank conditions (\ref{eq:rank}) form good defining equations of $\ol{O}_{r_1, r_2}$ in $X$.
\end{theorem}

\begin{proof}
We have a proper collapsing map $\pi: Z \to X$ given by projection, where $Z=\{(\phi, \mc{R}) \in X \times \op{Gr}(r_1,E) \mid \phi \in \mc{R}\oo F\}$. One can write $Z=G\times_P V$, where $P\subset G$ is a parabolic with Levi subgroup $L=\GL(\kk^{r_1})\times \Sp(F)$. Here, $L$ acts on $V=\kk^{r_1}\oo F$ naturally, while the unipotent radical of $P$ acts trivially. Let $\ol{O'}_{r_1,r_2} \subset V$ be the $L$-orbit closure given by the respective conditions (\ref{eq:rank}). Since $r_1\leq f$, we know by Proposition \ref{prop:sympinv} that $\ol{O'}_{r_1,r_2}$ is normal, and when $\cha \kk = 0$ has rational singularities (note that if $r_1=f$, then $r_2=r_1$ and  $\ol{O'}_{r_1,r_2}= V$).
Thus, the claim about the normality and rational singularities of $\ol{O}_{r_1,r_2}=G \cdot \ol{O'}_{r_1,r_2} = \pi(\ol{O'}_{r_1,r_2})$ follows from \cite[Theorems 1.3 and 1.4]{me} (see also \cite{kempf76}) and the claim on defining equations from \cite[Theorem 3.15]{me}.
\end{proof}

\begin{remark}
In the case $r_1=r_2$, the proof above gives the collapsing of $G\times_P V$ onto the determinantal varieties $\ol{O}_{r_1,r_2}$ in which case we can use  \cite[Theorems 1.2 and 1.3]{me}  (as done in \cite[Section 4.1]{me}) to get the stronger fact that they are strongly $F$-regular, which is well-known (see \cite{hh2}).
\end{remark}

By abuse of terminology, when $\cha \kk = 0$ we call a variety strongly $F$-regular if it is of strongly $F$-regular type. We mention a sharper result for special orbit closures that follows readily from \cite[Theorem 1.2]{me} (the proof will follow the notations therein).

\begin{proposition}\label{prop:special1}
The varieties $\ol{O}_{r,0}$ are strongly $F$-regular (thus, Cohen--Macaulay).
\end{proposition}

\begin{proof}
$\ol{O}_{r,0}$ can be realized as the image of a collapsing via the symplectic isotropic Grassmannian $\op{IGr}(r, F)$. Namely, let $Z=\{(\phi, \mc{R}) \in X \times \op{IGr}(r,F) \mid \phi \in E\oo \mc{R}\}$. Then the image of the (proper) projection map $Z \to X$ is $\ol{O}_{r,0}$. We can write $Z=G \times_P V$, where the Levi subgroup of $P$ can be identified with $L=\GL(E) \times \GL(\kk^{r})$ acting on $V=E\oo \kk^{r}$ naturally, and the unipotent radical of $P$ acts trivially on $V$. We have $G \cdot V = \ol{O}_{r,0}$, and $V$ is $L$-good by the Cauchy formula \cite[Theorem 2.3.2]{weymanbook} (or by Lemma \ref{lem:good}). The result follows by \cite[Theorems 1.2 and 1.3]{me}.
\end{proof}

In particular, when $r$ is maximal, i.e. $r=\min\{e,f/2\}$, $\ol{O}_{r,0}$ is the (reduced) nullcone, whose $F$-regularity was proved in \cite[Theorem 3.6]{vaib2} as well, using different methods. 

\begin{remark}\label{rem:infactFreg1}
Assume that $f\geq 2e$. Then in fact the orbit closures $\ol{O}_{e,r}$ are also strongly $F$-regular, strengthening Proposition \ref{prop:sympinv}. This can be seen as follows. 

We have $\codim_X \psi^{-1}(0) = \dim \bigwedge^2 E$, i.e. the nullcone is a complete intersection. This implies that all fibers of $\psi$ have the same dimension (e.g. see \cite[Section 8.1]{popvin}), and hence, the map $\psi$ is flat (e.g. \cite[Exercise III.10.9]{hartshorne}). Since $\psi^{-1}(0)=\ol{O}_{e,0}$ is $F$-rational, and so are the varieties $\ol{O}_r \subset \bigwedge^2 E$ \cite{baetica}, by \cite[Theorem 4.3]{aberenes} we get that $\ol{O}_{e,r} = \psi^{-1}(\ol{O}_r)$ is $F$-rational. Since it is Gorenstein by Proposition \ref{prop:sympinv}, it is also strongly $F$-regular by \cite[Corollary 4.7]{hh}.
\end{remark}

\section{Orthogonal case}

Throughout in this section we denote $G=\GL(E)\times \op{SO}(F)$, and $X=E\oo F$, and  $\cha \kk \neq 2$ (see the remark at the end for some comments in this case).

We note that $X$ is good (i.e. $\kk[X]$ has a good filtration), which follows again from Lemma \ref{lem:good} and \cite[Section 4.9]{andjan}.

We recall (see \cite[Corollary 2.4]{lovett05})
\begin{equation}\label{eq:codim2}
\codim_X \ol{O}_{r_1,r_2} =  (e-r_1)(f-r_1)+ \binom{r_1-r_2+1}{2}.
\end{equation}

We begin again with an easy result, which is the analogue of Lemma \ref{lem:smooth1}.

\begin{lemma}\label{lem:smooth2}
Let $K$ be a $b\times b$ invertible symmetric matrix, and $A$ an $a\times b$ matrix with $\rank A=a$. Then any $a\times a$ symmetric matrix $S$ can be written in the form $S=A K B^t + B K A^t$, for some $a \times b$ matrix $B$.
\end{lemma}

\begin{proof}
Let $Y$ be a left inverse of the matrix $K A^t$, i.e. $Y K A^t= I_a$. We can then put $B=1/2 \cdot SY$ (since $\cha \kk \neq 2$).
\end{proof}

Most of the results follow analogously to the symplectic case, therefore we will often only sketch the arguments. We explain in more detail why some exceptional cases occur.

\begin{proposition}\label{prop:orthinv}
Let $e < f$. Then for any $r$ with $2e-f\leq r\leq e$, the variety defined by the rank conditions $(e, r)$ in (\ref{eq:rank}) is Cohen--Macaulay; it is Gorenstein if and only if either $e-r$ is odd, or $r=0$, or $r=e$; it is normal if and only if $r\neq 2e-f$. Moreover, the minors corresponding to the isotropic rank conditions in (\ref{eq:rank}) yield good defining equations in $X$.
\end{proposition}

\begin{proof}
The proof is analogous to that of Proposition \ref{prop:sympinv}, but now we invoke the invariant theory of the orthogonal group \cite{deconproc}. The map (\ref{eq:psi}) is in fact the onto map
\[\psi: X \longrightarrow  X/\!\!/\SO(F) \cong \op{S}_2(E).\] 
By Lemma \ref{lem:smooth2}, the restriction of this map to the open subset of elements of full rank is again smooth. The subvariety of $\ol{O}_r \subset \op{S}_2(E)$ of symmetric matrices of rank $\leq r$ is Cohen--Macaulay \cite{kutz} (in fact strongly $F$-regular, see \cite[Section 4.1]{me}), and it is Gorenstein if and only if either $e-r$ is odd, or $r=0$, or $r=e$ (see \cite[Corollary 6.3.7]{weymanbook} -- note that by \cite[Theorem 4.4]{stanley} this holds in arbitrary characteristic since the Hilbert function is independent of that -- cf. \cite[Proposition 6.1.1]{weymanbook}).  Thus, the claims on the Cohen--Macaulay, Gorenstein, and good generator properties follow in the same way by Proposition \ref{prop:main} and  \cite[Section 4.1]{me} (see also Lemma \ref{lem:throwgens}). 

Denote by $Y$ be the variety defined by the rank conditions $(e,2e -f)$ in (\ref{eq:rank}) (which is equal to $\ol{O}_{e, 2e-f}$ unless $f=2e$). This case is not normal due to $\codim_{Y} \ol{O}_{e-1, 2e-f}=1$. It is easy to see using the Jacobian matrix \cite[Section I.5]{hartshorne} that a representative of the orbit $O_{e-1, 2e-f}$ is indeed in the singular locus of $Y$. Thus, $Y$ is not regular in codimension 1, hence not normal \cite[Theorem I.8.22A]{hartshorne}.
\end{proof}

In the following case we can make the strongest claims, analogous to Proposition \ref{prop:special1}.

\begin{proposition}\label{prop:special2}
The varieties $\ol{O}_{r,0}$ are strongly $F$-regular (thus, Cohen--Macaulay).
\end{proposition}

\begin{proof}
The proof is essentially the same as the proof of Proposition \ref{prop:special1}, with the orthogonal isotropic Grassmannian $\op{IGr}(r,F)$ replacing the role of the symplectic one. The only relevant difference is that when $f$ is even, $\op{IGr}(f/2,F)$ is disconnected. Then in the argument we just use the connected components of $\op{IGr}(f/2,F)$, creating two orbit closures under collapsing that are the distinct irreducible components of the variety given by the rank conditions $(f/2,0)$ in (\ref{eq:rank}). The claims follow the same way by \cite[Theorem 1.2]{me}.
\end{proof}

Again, when $r$ is maximal, then $\ol{O}_{r,0}$ is an irreducible component of the nullcone, and the above strengthens the results on Cohen--Macaulayness in \cite[Theorem 1.3]{vaib1}. 

\begin{remark}\label{rem:infactFreg2}
Assume that $f> 2e$, so the nullcone is a complete intersection and irreducible. As in Remark \ref{rem:infactFreg1}, we can say more about $\ol{O}_{e,r}$. The fact that they are $F$-rational follows analogously by \cite[Theorem 4.3]{aberenes}, \cite[Section 4.1]{me} and Proposition \ref{prop:special2}. When $e-r$ is odd,  $\ol{O}_{e,r}$ is Gorenstein by Proposition \ref{prop:orthinv}, hence also strongly $F$-regular by \cite[Corollary 4.7]{hh}. To obtain strong $F$-regularity for the case $e-r$ even as well (at least when $f> 2e+2$), one can use a collapsing over $\Gr(e,\kk^{e+1})$ as in the proof of the theorem below, and use the direct summand property established in \cite[Proposition 3.4]{me} together with \cite[Theorem 5.5]{hh}.
\end{remark}

\begin{theorem} \label{thm:mainorth}
The $\GL(E)\times \SO(F)$-orbit closures in $\ol{O}_{r_1, r_2}$ are normal if and only if either $r_2\neq 2 r_1 - f$ or $r_2=0$ or $r_2=e$. When  $\ol{O}_{r_1, r_2}$ is not normal, it is Cohen--Macaulay if and only if $r_1=e$. The minors corresponding to all the rank conditions (\ref{eq:rank}) give (good) generators of their defining ideals, unless $(r_1, r_2)=(f/2, 0)$ with $f$ even.
\end{theorem}

\begin{proof}
Note that the $r_2=0$ case follows from the previous proposition, and the case $r_2=e_1$ is trivial. Otherwise, we proceed by collapsing analogously to the proof of Theorem \ref{thm:mainsymp}. Note that the normality property is transferable through collapsing in both directions in \cite[Theorem 1.4]{me}, yielding the non-normal cases based on Proposition \ref{prop:orthinv}. But using (\ref{eq:codim2}), we see that if $r_1\neq e$ then these non-normal cases do become regular in codimension $1$, hence cannot be Cohen--Macaulay by \cite[Theorem I.8.22A]{hartshorne}.
\end{proof}

\begin{remark}
When $f=4$, the singularities of these orbit closures have been studied recently in characteristic $0$ via methods of local cohomology \cite[Corollary 6.2 and Remark 6.3]{lorper}.
\end{remark}

Finally, we give the (good) defining equations in the exceptional case $\ol{O}_{f/2,0}$ (with $f$ even) as well. Let us explain in simple terms why the rank conditions (\ref{eq:rank}) give two irreducible components in the essential case $f=2e$ first (cf. \cite[Definition 7.7]{vaib1}).

Take an $e\times f$ generic matrix of variables, written in bloc form as $[X \,\, Y]$, with $X, Y$ of size $e \times e$. Then the symplectic rank condition in (\ref{eq:rank}) is equivalent to $X\cdot X^t + Y\cdot Y^t = 0$. Taking determinants, we obtain the two equations
\[\det(X) =\pm i^e \cdot \det(Y),\]
neither of which is in the ideal generated by the minors in (\ref{eq:rank}). We further take the $\SO(F)$-spans of the two equations above yielding the decomposition of the space spanned by the maximal minors into the direct sum of $G$-modules:
\[\bigwedge^e E^* \oo \bigwedge^e F= V_+ \oplus V_- \, \subset \kk[X].\]
In fact, this stems from the well-known fact that while $\bigwedge^{f/2} F$ is irreducible as an $\op{O}(F)$-module, it splits into two irreducibles $\bigwedge^{f/2} F=M_+ \oplus M_-$ as $\SO(F)$-modules. Thus, $V_\pm = \bigwedge^e E^*  \oo M_\pm$, as $G$-modules. 

More generally when $f\leq 2e$, write $\ol{O}_{f/2, 0}^{\pm}$ for the respective $G$-orbit closure vanishing on $\bigwedge^{f/2} E^* \oo M_\pm \, \subset \kk[X]$ and satisfying the rank conditions (\ref{eq:rank}).

\begin{proposition}\label{prop:eqred}
The minors corresponding to all the rank conditions (\ref{eq:rank}) together with $\bigwedge^{f/2} E^* \oo M_\pm \, \subset \kk[X]$ give (good) defining equations of $\ol{O}_{f/2, 0}^{\pm}$ in $X$.
\end{proposition}

\begin{proof}
First, consider the case $f=2e$. Put $W=\op{S}_2(E^*) \oo \op{triv} \, \subset \kk[X]$, which is spanned by the quadratic $\SO(F)$-invariants. Let $I_\pm$ be the ideals generated by $W \oplus V_\pm$. By the discussion above, the zero-set of $I_\pm$ is $\ol{O}_{e, 0}^{\pm}$. Since the nullcone equals $\ol{O}_{e, 0}^+\bigcup \ol{O}_{e, 0}^{-}$, \, $W$ gives good generators for the (radical) ideal $I_+\cap I_-$, as seen in Proposition \ref{prop:orthinv}. Furthermore, by Theorem \ref{thm:mainorth} we have that $W\oplus V_+ \oplus V_-$ gives good defining equations of $\ol{O}_{e-1,0} = \ol{O}_{e, 0}^+\bigcap \ol{O}_{e, 0}^{-}$. Thus, by Lemma \ref{lem:reducible}, $W\oplus V_\pm$ gives good defining equations of $I_\pm$. Moreover, since $I_+\cap I_-$ and $I_+ + I_-$ are reduced, so are $I_\pm$.

The general case $f\leq 2e$ follows from this by collapsing, as in the proof of Theorem \ref{thm:mainorth}, using \cite[Theorem 3.15]{me} again.
\end{proof}

\begin{remark}\label{rem:errors2}
The papers \cite{lovett05, lovett07} mistakenly claim that the rank conditions (\ref{eq:rank}) give precisely all of the $G$-orbit closures, not mentioning the exceptional case $(f/2,0)$ that is described in Proposition \ref{prop:eqred} above.
\end{remark}

\begin{remark}\label{rem:char2}
When $\cha \kk = 2$, $X$ is not good as a $G = \GL(E)\times \op{SO}(F)$-module by \cite[Example 1.]{hashiinv} (for $f$ even). Furthermore, some extra orbits appear as the natural representation $F$ is not an irreducible $\op{SO}(F)$-module when $f$ is odd, due to the bilinear form being degenerate. One additional detail is that one has to further consider the second divided power in (\ref{eq:psi}) instead of $\op{S}_2(E)$ -- see also Lemma \ref{lem:smooth2}. Due to such exceptional phenomena, we leave this characteristic for a separate consideration.
\end{remark}

\subsection*{Acknowledgments}
The author would like to express his gratitude to Jerzy Weyman for his valuable comments and suggestions, as well to the anonymous referee for carefully reading the paper and detailed feedback which helped us significantly improve the manuscript.

\bibliographystyle{alpha}
\bibliography{biblo}

\end{document}